\documentclass{article}
\usepackage{amsfonts}

\newtheorem{theorem}{Theorem}

\newtheorem{definition}[theorem]{Definition}
\newtheorem{example}[theorem]{Example}

\newtheorem{lemma}[theorem]{Lemma}

\newtheorem{remark}[theorem]{Remark}

\newenvironment{proof}[1][Proof]{\noindent\textbf{#1.} }{\ \rule{0.5em}{0.5em}}
\input{tcilatex}
\begin{document}

\title{Structures in Concrete Categories}
\author{Henri Bourl\`{e}s\thanks{%
Satie, ENS de Cachan/CNAM, 61 Avenue Pr\'{e}sident Wilson, F-94230 Cachan,
France (henri.bourles@satie.ens-cachan.fr) }}
\maketitle

\begin{abstract}
The monumental treatise "\'{E}l\'{e}ments de math\'{e}matique" of N.
Bourbaki is based on the notion of structure and on the theory of sets. \ On
the other hand, the theory of categories is based on the notions of morphism
and functor. An appropriate definition of a structure in a concrete category
is given in this paper. This definition makes it possible to bridge the gap
between Bourbaki's structural approach and the categorical language.
\end{abstract}

\section{Introduction}

The monumental treatise "\'{E}l\'{e}ments de math\'{e}matique" of N.
Bourbaki is based on the notion of structure and on the theory of sets \cite%
{Bourbaki-E}. \ On the other hand, the theory of categories is based on the
notions of morphism and functor. \ The notion of "concrete category", as
developed in \cite{Adamek}, did not bridge the gap between the categorical
approach and Bourbaki's structures. According to \cite{Adamek}, a structure
in a concrete category $\left( \mathcal{C},\left\vert .\right\vert ,\mathcal{%
X}\right) ,$ where $\mathcal{X}$ is the base category and $\left\vert
.\right\vert :\mathcal{C}\rightarrow \mathcal{X}$ is the forgetful functor,
is an object $A$ of $\mathcal{C}$. \ If $A,$ $A^{\prime }$ are two $\mathcal{%
C}$-isomorphic objects such that $\left\vert A\right\vert =\left\vert
A^{\prime }\right\vert ,$ these objects are not equal in general. \
Nevertheless, if $\mathcal{C}$ is, for example, the category $\mathbf{Top}$
of topological spaces and $\mathcal{X}$ is the category $\mathbf{Set}$ of
sets, $A$ and $A^{\prime }$ are two equal sets endowed with the same
topologies.

\begin{definition}
A structure in $\left( \mathcal{C},\left\vert .\right\vert ,\mathcal{X}%
\right) $ is an equivalence class of objects of $\mathcal{C}$ with respect
to the following equivalence relation: two objects $A$ and $A^{\prime }$ of $%
\mathcal{C}$ are equivalent if $\left\vert A\right\vert =\left\vert
A^{\prime }\right\vert $ and there is a $\mathcal{C}$-isomorphism $%
A\rightarrow A^{\prime }.$
\end{definition}

Some consequences of this definition are studied the sequel.

\section{Structures}

\subsection{Isomorphisms of structures}

Let $\mathfrak{S}_{1}$ and $\mathfrak{S}_{2}$\ be two structures in $\left( 
\mathcal{C},\left\vert .\right\vert ,\mathcal{X}\right) .$ \ Assume that
there exist representatives $A_{1}$ and $A_{2}$ of $\mathfrak{S}_{1}$ and $%
\mathfrak{S}_{2},$ respectively, which are isomorphic in $\mathcal{C}$. \
Then $\left\vert A_{1}\right\vert $ and $\left\vert A_{2}\right\vert $ are
not necessarily equal, but are isomorphic in $\mathcal{X}$. \ If $%
A_{1}^{\prime }$ and $A_{2}^{\prime }$ are other representatives of $%
\mathfrak{S}_{1}$ and $\mathfrak{S}_{2},$ they are isomorphic in $\mathcal{C}
$.

\begin{definition}
The structures $\mathfrak{S}_{1}$ and $\mathfrak{S}_{2}$\ are called
isomorphic if the above condition holds.
\end{definition}

This definition determines an equivalence relation on structures in $\left( 
\mathcal{C},\left\vert .\right\vert ,\mathcal{X}\right) $.

\subsection{Comparison of structures}

Obviously, if $\mathfrak{S}_{1}$ and $\mathfrak{S}_{2}$\ are two structures
in $\left( \mathcal{C},\left\vert .\right\vert ,\mathcal{X}\right) $, and $%
A_{1}$ and $A_{2}$ are representatives of $\mathfrak{S}_{1}$ and $\mathfrak{S%
}_{2},$ respectively, such that $\left\vert A_{1}\right\vert =\left\vert
A_{2}\right\vert $ and there exists a $\mathcal{C}$-morphism $%
A_{1}\rightarrow A_{2}$ such that $\left\vert A_{1}\rightarrow
A_{2}\right\vert =\func{id}_{A_{1}},$ then for any representatives $%
A_{1}^{\prime }$ and $A_{2}^{\prime }$ of $\mathfrak{S}_{1}$ and $\mathfrak{S%
}_{2},$ respectively, there exists a $\mathcal{C}$-morphism $A_{1}^{\prime
}\rightarrow A_{2}^{\prime }$ such that $\left\vert A_{1}\rightarrow
A_{2}\right\vert =\func{id}_{A_{1}^{\prime }}.$

\begin{definition}
If the above condition holds, then the structure $\mathfrak{S}_{1}$ is
called finer than $\mathfrak{S}_{2}$ (written $\mathfrak{S}_{1}\geq 
\mathfrak{S}_{2}$) and $\mathfrak{S}_{2}$ is called coarser than $\mathfrak{S%
}_{2}$ (written $\mathfrak{S}_{2}\leq \mathfrak{S}_{1}$).
\end{definition}

\begin{lemma}
\label{lemma-order-relation}The relation $\geq $ on structures in $\left( 
\mathcal{C},\left\vert .\right\vert ,\mathcal{X}\right) $ is an order
relation.
\end{lemma}

\begin{proof}
The reflexivity and the transitivity are obvious. \ Assume that $\mathfrak{S}%
_{1}\geq \mathfrak{S}_{2}$ and $\mathfrak{S}_{2}\geq \mathfrak{S}_{1}.$ \
Let $A_{1},A_{2}$ be representatives of $\mathfrak{S}_{1},\mathfrak{S}_{2}$
respectively. \ There exist $\mathcal{C}$-morphisms $A_{1}\rightarrow A_{2}$
and $A_{2}\rightarrow A_{1}$ such that $\left\vert A_{1}\rightarrow
A_{2}\right\vert =\func{id}_{\left\vert A_{1}\right\vert }$ and $\left\vert
A_{2}\rightarrow A_{1}\right\vert =\func{id}_{\left\vert A_{2}\right\vert }.$
\ Thus $A_{1}\rightarrow A_{2}\longrightarrow A_{1}$ and $A_{2}\rightarrow
A_{1}\rightarrow A_{2}$ are $\mathcal{C}$-morphisms such that $\left\vert
A_{1}\rightarrow A_{2}\longrightarrow A_{1}\right\vert =\func{id}%
_{\left\vert A_{1}\right\vert }$ and $\left\vert A_{2}\rightarrow
A_{1}\rightarrow A_{2}\right\vert =\func{id}_{A_{2}}.$ \ Since the forgetful
functor is faithful, it follows that $A_{1}\rightarrow A_{2}\longrightarrow
A_{1}=\func{id}_{A_{1}}$ and $A_{2}\rightarrow A_{1}\rightarrow A_{2}=\func{%
id}_{A_{2}}.$ \ Therefore, $A_{1}\rightarrow A_{2}$ and $A_{2}\rightarrow
A_{1}$ are isomorphisms, and each of them is the inverse of the other. \ As
a result, $\mathfrak{S}_{1}=\mathfrak{S}_{2}.$
\end{proof}

\begin{example}
A structure in the concrete category $\left( \mathbf{Top},\left\vert
.\right\vert ,\mathbf{Set}\right) $ is a topology. \ The structure $%
\mathfrak{S}_{1}$ is finer than $\mathfrak{S}_{2}$ if, and only if the
topology $\mathfrak{S}_{1}$ is finer than the topology $\mathfrak{S}_{2}.$
\end{example}

\subsection{Initial structures and final structures}

Let $A,B$ be objects of $\mathcal{C}$. \ Recall that the relation "$%
f:\left\vert B\right\vert \rightarrow \left\vert A\right\vert $ is a $%
\mathcal{C}$-morphism" means that there exists a morphism $f:B\rightarrow A$
(necessarily unique) such that $\left\vert B\rightarrow A\right\vert
=\left\vert B\right\vert \rightarrow \left\vert A\right\vert $ (\cite{Adamek}%
, Remark 6.22).

Let $\left( f_{i}:A\rightarrow A_{i}\right) _{i\in I}$ be a source in $%
\mathcal{C}$, i.e. a family of $\mathcal{C}$-morphisms. \ This source is
called initial if for each object $B$ of $\mathcal{C},$ the relation

$\qquad $"$f:\left\vert B\right\vert \rightarrow \left\vert A\right\vert $
is a $\mathcal{C}$-morphism"

is equivalent to the relation

$\qquad $"for all $i\in I,$ $f_{i}\circ f:\left\vert B\right\vert
\rightarrow \left\vert A_{i}\right\vert $ is a $\mathcal{C}$-morphism".

Obviously, the source $\left( f_{i}:A\rightarrow A_{i}\right) _{i\in I}$ is
initial if, and only if for every object $A^{\prime }$ with the same
structure as $A,$ i.e. for which there exists a $\mathcal{C}$-isomorphism $%
\varphi :A^{\prime }\rightarrow A$ with $\left\vert \varphi \right\vert =%
\limfunc{id}_{\left\vert A\right\vert }$, the source $\left( f_{i}\circ
\varphi :A^{\prime }\rightarrow A_{i}\right) _{i\in I}$ is initial.

\begin{definition}
If the source $\left( f_{i}:A\rightarrow A_{i}\right) _{i\in I}$ is initial,
then the structure of $A$ is called initial with respect to the family $%
\left( f_{i}:\left\vert A\right\vert \rightarrow \left\vert A_{i}\right\vert
\right) _{i\in I}.$
\end{definition}

\begin{theorem}
If the structure $\mathfrak{S}$ of $A$ is initial for the family $\left(
f_{i}:\left\vert A\right\vert \rightarrow \left\vert A_{i}\right\vert
\right) _{i\in I},$ $\mathfrak{S}$ is the coarsest of all structures $%
\mathfrak{S}^{\prime }$ of objects\ $A^{\prime }$ of $\mathcal{C}$ such that
all $\mathcal{X}$-morphisms $f_{i}:\left\vert A^{\prime }\right\vert
\rightarrow \left\vert A_{i}\right\vert $ $\left( i\in I\right) $\ are $%
\mathcal{C}$-morphisms, and consequently is unique.
\end{theorem}

\begin{proof}
Assume that the structure $\mathfrak{S}$ of $A$ is initial for the family $%
\left( f_{i}:\left\vert A\right\vert \rightarrow \left\vert A_{i}\right\vert
\right) _{i\in I}.$ \ Then all $f_{i}:\left\vert A\right\vert \rightarrow
\left\vert A_{i}\right\vert $ $\left( i\in I\right) $\ are $\mathcal{C}$%
-morphisms. \ Let $A^{\prime }$ be such all $f_{i}:\left\vert A^{\prime
}\right\vert \rightarrow \left\vert A_{i}\right\vert $ are $\mathcal{C}$%
-morphisms. \ Since $f_{i}$ is an $\mathcal{X}$-morphism, $\left\vert
A^{\prime }\right\vert =\left\vert A\right\vert $ and $f=\limfunc{id}%
_{\left\vert A\right\vert }$ is an $\mathcal{X}$-morphism and for each $i\in
I,$ $f_{i}\circ f:\left\vert A^{\prime }\right\vert \rightarrow \left\vert
A_{i}\right\vert =f_{i}$ is a $\mathcal{C}$-morphism. \ Thus $f:\left\vert
A^{\prime }\right\vert \rightarrow \left\vert A\right\vert $ is a $\mathcal{C%
}$-morphism and $\mathfrak{S}$ is coarser than the structure of $A^{\prime
}. $ \ This determines $\mathfrak{S}$ in a unique way by Lemma \ref%
{lemma-order-relation}.
\end{proof}

As shown by (\cite{Bourbaki-E}, p. IV.30, exerc. 30), it may happen that a
structure $\mathfrak{S}$ in a concrete category $\left( \mathcal{C}%
,\left\vert .\right\vert ,\mathcal{X}\right) $ be the coarsest of all
structures $\mathfrak{S}^{\prime }$ of objects\ $A^{\prime }$ of $\mathcal{C}
$ such all $\mathcal{X}$-morphisms $f_{i}:\left\vert A^{\prime }\right\vert
\rightarrow \left\vert A_{i}\right\vert $ are $\mathcal{C}$-morphisms, and
that this structure be not initial with respect to the family $\left(
f_{i}:\left\vert A\right\vert \rightarrow \left\vert A_{i}\right\vert
\right) _{i\in I}.$

The opposite of the concrete category $\left( \mathcal{C},\left\vert
.\right\vert ,\mathcal{X}\right) $ is $\left( \mathcal{C}^{\limfunc{op}%
},\left\vert .\right\vert ^{\limfunc{op}},\mathcal{X}^{\limfunc{op}}\right) $
where for each pair of objects $A,B$ of $\mathcal{C}$, and any morphism $%
A\longleftarrow B$ in $\mathcal{C}^{\limfunc{op}},\ \left\vert
A\longleftarrow B\right\vert ^{\limfunc{op}}=\left\vert A\leftarrow
B\right\vert .$ \ The structure of $A$ is called final with respect to the
family $\left( f_{i}:\left\vert A\right\vert \longleftarrow \left\vert
A_{i}\right\vert \right) _{i\in I}$ in $\left( \mathcal{C},\left\vert
.\right\vert ,\mathcal{X}\right) $\ if it is initial with respect to this
family in $\left( \mathcal{C}^{\limfunc{op}},\left\vert .\right\vert ^{%
\limfunc{op}},\mathcal{X}^{\limfunc{op}}\right) .$

\section{Applications}

Examples of initial structures (\cite{Bourbaki-E}, \S\ 2.4): inverse image
of a structure, induced structure, concrete product (see Remark \ref%
{remark-product}).

Examples of final structures (\cite{Bourbaki-E}, \S\ 2.6): direct image of a
structure, quotient structure, concrete coproduct (see Remark \ref%
{remark-product}).

\begin{remark}
\label{remark-product}Let $\mathcal{P}=\left( \limfunc{pr}_{i}:A\rightarrow
A_{i}\right) _{i\in I}$ be a product in $\mathcal{C}.$ \ This product is
called concrete in $\left( \mathcal{C},\left\vert .\right\vert ,\mathcal{X}%
\right) $ if $\left\vert \mathcal{P}\right\vert =\left( \left\vert \limfunc{%
pr}_{i}\right\vert :\left\vert A\right\vert \rightarrow \left\vert
A_{i}\right\vert \right) _{i\in I}$ is a product in $\mathcal{X}$ (\cite%
{Adamek}, \S 10.52). \ All products considered in (\cite{Bourbaki-E}, \S\ %
2.4) are concrete but there exist concrete categories with non-concrete
products (\cite{Adamek}, \S 10.55(2)). \ Dually, there exist concrete
categories with non-concrete coproducts (\cite{Adamek}, \S 10.67).
\end{remark}

\end{document}